\documentclass[preprint,1p,times]{elsarticle}
\usepackage{amsmath, amssymb, amsthm}
\usepackage{graphicx}
\usepackage{color,soul}
\usepackage{float}
\usepackage{lineno}
\usepackage{pgf}
\usepackage{epstopdf}

\journal{ }

\theoremstyle{definition}
\newtheorem{theorem}{Theorem}           
\newtheorem{lemma}{Lemma}               

\newtheorem{definition}{Definition}

\newtheorem{remark}{Remark}

\begin{document}

\begin{frontmatter}
	
\title{Dynamics of vertically and horizontally transmitted parasites: Continuous vs Discrete Models}
\author[cmbe]{Priyanka Saha\fnref{fn1}}
\author[cmbe]{Nandadulal Bairagi\corref{cor1}}
\ead{nbairagi.math@jadavpuruniversity.in}
\cortext[cor1]{Corresponding author}
%\fnteyt[fn2]{Research is supported by PURSE DST Phase II.}
%\address[rmv]{Department of Mathematics\\ Rabindra Mahavidyalaya, Champadanga, Hooghly-712401\\ West Bengal, India.}
\address[cmbe]{Centre for Mathematical Biology and Ecology\\ Department of Mathematics, Jadavpur University\\ Kolkata-700032, India.}

\begin{abstract}
In this paper we analyze a continuous-time epidemic model and its discrete counterpart, where infection spreads both horizontally and vertically. We consider three cases: model with horizontal and imperfect vertical transmissions, model with horizontal and perfect vertical transmissions, and model with perfect vertical and no horizontal transmissions. Stability of different equilibrium points of both the continuous and discrete systems in all cases are determined. It is shown that the stability criteria are identical for continuous and discrete systems. The dynamics of the discrete system have also shown to be independent of the step-size. Numerical computations are presented to illustrate analytical results of both the systems and their subsystems.
\end{abstract}


\end{frontmatter}

\section{Introduction}
Mathematical models play significant role in understanding the dynamics of biological phenomena. System of nonlinear differential equations are frequently used to describe these biological models. Unfortunately, nonlinear differential equations in general can not be solved analytically. For this reason, we go for numerical computations of the model system and discretization of the continuous model is essential in this process. Here we shall explore and compare the dynamics of a continuous time epidemic model and its subsystems with their corresponding discrete models.

Lipsitch et al.\ \cite{LNEM95} have investigated the dynamics of vertically and horizontally transmitted parasites of the following population model, where the state variables $X$ and $Y$ represent, respectively, the densities of uninfected and infected hosts at time $t$: 
\begin{eqnarray}\label{continuous_1}
\frac{dX}{dt} & =&  \left[b_x\left\{1-\frac{(X+Y)}{K}\right\}-u_x-\beta Y\right]X+e\left\{1-\frac{(X+Y)}{K}\right\}Y\mbox{,}\\
\frac{dY}{dt} & = & \left[b_y\left\{1-\frac{(X+Y)}{K}\right\}-u_y+\beta X\right]Y\mbox{.}\nonumber
\end{eqnarray}
This model demonstrates the rate equations of a density dependent asexual host populations, where infection spreads through imperfect vertical transmission as well as horizontal transmission. Horizontal transmission of infection follows mass action law with $\beta$ as the proportionality constant. Vertical transmission is imperfect because infected hosts not only give birth of infected hosts at a rate $b_y$ but also produce uninfected offspring at a rate $e$. In case of perfect vertical transmission, however, infected hosts give birth of infected hosts only and in that case $e=0$. Parasites may affect the fecundity and morbidity rates of its host population \cite{HB72,LM96}. It is assumed here that the death rate of infected hosts is higher than that of susceptible hosts, i.e., $u_y>u_x$ and the birth rate of susceptible hosts is higher than that of infected hosts, i.e., $b_x\geq b_y+e$.

Standard finite difference schemes, such as Euler method, Runge--Kutta method etc.\ are frequently used for numerical solutions of both ordinary and partial differential equations. But the behavior of standard finite difference schemes depend heavily on the step size. They fail to preserve positivity of the solutions for all step size. These conventional discretized models also show numerical instability and exhibit spurious behaviors like chaos which are not observed in the corresponding continuous models. In other words, these discrete models are dynamically inconsistent. So it becomes important to construct discrete models which will preserve all the properties of its constituent continuous models without any restriction on the step size. Mickens in 1989 first proposed such nonstandard finite difference (NSFD) scheme \cite{M89} and was shown to have identical dynamics with its corresponding continuous model. It was also demonstrated that the dynamics is completely independent of step size. Successful application of this technique in subsequent time is observed in different biological models \cite{MET03,BB17a,SI10,BB17b,RL13,BET17,G12,BB17}. Here we will discretize a continuous time population model in which parasite transmitted both vertically and horizontally following dynamics preserving nonstandard finite difference (NSFD) method introduced by Mickens \cite{M89}. We present the local stability analysis of both the continuous and discrete systems and prove that the dynamic behaviour of both systems are identical with same parameter restrictions. Moreover, we prove that the proposed discrete models are positive for all step size and dynamically consistent.

The paper is organized as follows. We present the analysis of the continuous time model and its subsystems in the next section. Section 3 contains the corresponding analysis in discrete system. In section 4, we present extensive numerical simulations in favour of our theoretical results. Finally, a summary is presented in section 5.
\section{Analysis of Continuous Time Models}
Lipsitch et al.\ \cite{LNEM95} analyzed the system \eqref{continuous_1} with perfect vertical transmissions (the case $e=0$) and horizontal transmission as well as the system with perfect vertical transmission but no horizontal transmission (the case $\beta=0$, $e=0$). The general case (when $e\neq0$, $\beta\neq0$) was analyzed numerically and its importance in the prevalence of infection was discussed. Here we first give stability analysis of equilibrium points of the general model \eqref{continuous_1} and deduce the results of subcases, whenever applicable.

The continuous system \eqref{continuous_1} has two boundary equilibrium points $E_0=(0,0)$, $E_1=(\bar{X},0)$, where $\bar{X}=K\left(1-\frac{u_x}{b_x}\right)$ and one interior equilibrium point $E^*=(X^*,Y^*)$, where the equilibrium densities of susceptible and infected hosts are given by\\
$~~~~~~~~~~~~~~~~X^*=\frac{-B+\sqrt{B^2-4AC}}{2A}$ and $Y^*  =  \frac{(\beta K-b_y)X^*}{b_y}+\frac{K(b_y-u_y)}{b_y}$,\\
%$$ X^*  =  \frac{-B+\sqrt{B^2-4AC}}{2A}~ and~
%Y^*  =  \frac{(\beta K-b_y)X^*}{b_y}+\frac{K(b_y-u_y)}{b_y},$$
with
\begin{eqnarray}\label{interior coefficient_1}
\left\{
\begin{array}{ll}
A  = \frac{\beta K}{b_y^2}\{b_y(b_x-b_y-e)+\beta K(b_y+e)\}\mbox{,}\\
B  = -K(b_x-u_x)+K(b_x+\beta K+e)\frac{(b_y-u_y)}{b_y}+2eK\frac{(\beta K-b_y)(b_y-u_y)}{b_y^2}-\frac{eK(\beta K-b_y)}{b_y}\mbox{,}\\
C  = -\frac{eK^2(b_y-u_y)u_y}{b_y^2}\mbox{.}
\end{array}
\right.
\end{eqnarray}
The trivial equilibrium $E_0$ exists for all parameter values, but the infection free equilibrium $E_1$ exists if $b_x>u_x$. The coexisting equilibrium point $E^*$ exists if $b_x>u_x$, $b_y>u_y$, $b_y>\beta K$ and $\frac{K}{X^*}>\frac{b_y-\beta K}{b_y-u_y}$.

We have the following theorem for the stability of different equilibrium points.
\begin{theorem}\label{cont_theorem_1}
	System \eqref{continuous_1} is locally asymptotically stable around the equilibrium point
	\begin{itemize}
		\item[(i)] $E_0$ if $b_x<u_x$ and $b_y<u_y$.
		\item[(ii)] $E_1$ if $b_x>u_x$ and $R_0<1$, where $R_0=V_0+H_0$ with $V_0=\frac{b_y}{b_x}\frac{u_x}{u_y}$, $H_0=\frac{\beta}{u_y}K\left(1-\frac{u_x}{b_x}\right)$ and it is unstable whenever $R_0>1$.
		\item[(iii)] $E^*$ if $b_x>u_x$, $b_y>u_y$, $b_y>\beta K$ and $\frac{K}{X^*}>\frac{b_y-\beta K}{b_y-u_y}$.
	\end{itemize}
\end{theorem}
\begin{proof}
	Local stability of the system around an equilibrium point is performed following linearization technique. For it, we compute the variational matrix of system \eqref{continuous_1} at an arbitrary fixed point $(X,Y)$ as
	\begin{eqnarray}\label{cont_variational_1}
	V(X,Y)=
	\begin{pmatrix}
	a_{11} & a_{12}\\
	a_{21} & a_{22}
	\end{pmatrix}\mbox{,}
	\end{eqnarray}
	where
	\begin{eqnarray}\label{cont_vari_element_1}
	\left\{
	\begin{array}{ll}
	a_{11} = b_x\left\{1-\frac{(X+Y)}{K}\right\}-\frac{b_xX}{K}-u_x-\beta Y-\frac{eY}{K}\mbox{,}\\
	a_{12} = -\frac{b_xX}{K}-\beta X+e\left\{1-\frac{(X+Y)}{K}\right\}-\frac{eY}{K}\mbox{,}\\
	a_{21} = -\frac{b_yY}{K}+\beta Y\mbox{,}\\
	a_{22} = b_y\left\{1-\frac{(X+Y)}{K}\right\}-\frac{b_yY}{K}-u_y+\beta X\mbox{.}
	\end{array}
	\right.
	\end{eqnarray}
	At the trivial fixed point $E_0$, the variational matrix is
	\begin{eqnarray}
	V(E_0)=
	\begin{pmatrix}
	b_x-u_x & e\\
	0 & b_y-u_y\nonumber
	\end{pmatrix}\mbox{.}
	\end{eqnarray}
	Corresponding eigenvalues are given by $\lambda_1=b_x-u_x$ and $\lambda_2=b_y-u_y$. $E_0$ will be locally asymptotically stable if $\lambda_1=b_x-u_x<0$ and $\lambda_2=b_y-u_y<0$; i.e., if $b_x<u_x$ and $b_y<u_y$. Thus, if the birth rates of susceptible hosts and infected hosts are less than their respective death rates, then both populations goes to extinction and the trivial equilibrium will be stable.\\
	At the axial equilibrium point $E_1$, the variational matrix is computed as
	\begin{eqnarray}
	V(E_1)=
	\begin{pmatrix}
	-b_x\left(1-\frac{u_x}{b_x}\right) & -(b_x+\beta K)\left(1-\frac{u_x}{b_x}\right)+\frac{eu_x}{b_x}\\
	0 & \frac{b_yu_x}{b_x}-u_y+\beta K\left(1-\frac{u_x}{b_x}\right)\nonumber
	\end{pmatrix}\mbox{.}
	\end{eqnarray}
	The corresponding eigenvalues are $\lambda_1=-b_x\left(1-\frac{u_x}{b_x}\right)$ and $\lambda_2=\frac{b_yu_x}{b_x}-u_y+\beta K\left(1-\frac{u_x}{b_x}\right)$. It is to be noted that $\lambda_1<0$ whenever $E_1$ exists. The other eigenvalue can be rearranged as $\lambda_2=u_y\left\{\frac{b_y}{b_x}\frac{u_x}{u_y}+\frac{\beta}{u_y}K\left(1-\frac{u_x}{b_x}\right)-1\right\}$. Thus $\lambda_2<0$ whenever $R_0<1$, where $R_0=V_0+H_0$. Note that $V_0=\left(\frac{b_y}{b_x}\right)\left(\frac{u_x}{u_y}\right)$ is the basic reproduction number due to vertical transmission and $H_0=\frac{\beta}{u_y}\bar{X}$ is the basic reproduction number due to horizontal transmission.\\
	At the interior equilibrium point $E^*$, the variational matrix is
	\begin{eqnarray}\label{cont_vari_int_1}
	V(E^*)=
	\begin{pmatrix}
	a_{11}^* &  a_{12}^*\\
	a_{21}^* & a_{22}^*
	\end{pmatrix}\mbox{,}
	\end{eqnarray}
	where
	\begin{eqnarray}\label{cont_vari_int_ele_1}
	\left\{
	\begin{array}{ll}
	a_{11}^* = -\frac{eY^*}{X^*}\left\{1-
	\frac{(X^*+Y^*)}{K}\right\}-\frac{b_xX^*}{K}-\frac{eY^*}{K}\mbox{,}\\
	a_{12}^* = -\frac{b_xX^*}{K}-\beta X^*+e\left\{1-\frac{(X^*+Y^*)}{K}\right\}-\frac{eY^*}{K}\mbox{,}\\
	a_{21}^* = -\frac{b_yY^*}{K}+\beta Y^*\mbox{,}\\
	a_{22}^* = -\frac{b_yY^*}{K}\mbox{.}
	\end{array}
	\right.
	\end{eqnarray}
	$E^*$ will be stable if and only if $Trace(V(E^*))<0$ and $Det(V(E^*))>0$. From the existence condition of $E^*$, one can observe that both $a_{11}^*$ and $a_{22}^*$ are negative. Thus, $Trace(V(E^*))<0$. After some simple algebraic manipulation, one gets $$Det(V(E^*))=\frac{eu_yY^*}{X^*}\left(1-\frac{u_y}{b_y}\right)+\frac{\beta X^*Y^*}{K}(b_x-b_y-e)+\beta^2X^*Y^*+\frac{e\beta^2X^*Y^*}{b_y}\mbox{.}$$
	Thus, whenever $E^*$ exists and $b_x \geq b_y+e$, we have $Det(V(E^*))>0$ and $E^*$ becomes locally asymptotically stable. This completes the proof.
\end{proof}
%\begin{corollary}
\subsection{Model with Horizontal and Perfect Vertical Transmissions}
The vertical transmission is perfect if infected hosts give birth to infected offspring only. In this case $e=0$ and the system \eqref{continuous_1} becomes
\begin{eqnarray}\label{continuous_2}
\frac{dX}{dt} &=& b_xX\left\{1-\frac{(X+Y)}{K}\right\}-u_xX-\beta XY\mbox{,}\\
\frac{dY}{dt} &=& b_yY\left\{1-\frac{(X+Y)}{K}\right\}-u_yY+\beta XY\mbox{.}\nonumber
\end{eqnarray}
The continuous system \eqref{continuous_2} has four equilibrium points, viz.\ $E_0^H=(0,0)$, $E_1^H=(\bar{X},0)$, $E_2^H=(0,\bar{Y})$ and interior equilibrium point $E_H^*=(X_H^*,Y_H^*)$, where
%$\bar{X}=K\left(1-\frac{u_x}{b_x}\right)$, $\bar{Y}=K\left(1-\frac{u_y}{b_y}\right)~\mbox{and}\nonumber\\
%X_H^*  &=&  \frac{b_xu_y-b_yu_x-\beta K(b_y-u_y)}{\beta(\beta K+b_x-b_y)}, ~
%Y_H^*  =  \frac{b_yu_x-b_xu_y+\beta K(b_x-u_x)}{\beta(\beta K+b_x-b_y)}.
\begin{eqnarray}
\bar{X}&=&K\left(1-\frac{u_x}{b_x}\right)\mbox{,} ~\bar{Y}=K\left(1-\frac{u_y}{b_y}\right) ~\mbox{and}\nonumber\\
X_H^*&=&\frac{b_xu_y-b_yu_x-\beta K(b_y-u_y)}{\beta(\beta K+b_x-b_y)}\mbox{,} ~
Y_H^*  =  \frac{b_yu_x-b_xu_y+\beta K(b_x-u_x)}{\beta(\beta K+b_x-b_y)}\mbox{.}\nonumber
\end{eqnarray}
The trivial equilibrium point $E_0^H$ always exists, $E_1^H$ exists if $b_x>u_x$, $E_2^H$ exists if $b_y>u_y$ and the interior fixed point $E_H^*$ exists if $b_x>u_x$, $b_y>u_y$, $\frac{b_xu_y}{b_y}>u_x+\beta K\left(1-\frac{u_y}{b_y}\right)$ and $R_0>1$, where $R_0=V_0+H_0$ with $V_0=\frac{b_y}{b_x}\frac{u_x}{u_y}$, $H_0=\frac{\beta}{u_y}K\left(1-\frac{u_x}{b_x}\right)$. The following results are known \cite{LNEM95}.
\begin{theorem}\label{cont_theorem_2}
	System \eqref{continuous_2} is locally asymptotically stable around the equilibrium point
	\begin{itemize}
		\item[(i)] $E_0^H$ if $b_x<u_x$ and $b_y<u_y$,
		\item[(ii)] $E_1^H$ if $b_x>u_x$ and $R_0<1$, %$$H_0=\frac{\beta}{u_y}K(1-\frac{u_x}{b_x})~and~ V_0=\frac{b_y}{b_x}\frac{u_x}{u_y},$$
		\item[(iii)] $E_2^H$ if $b_y>u_y$ and $\frac{b_xu_y}{b_y}<u_x+\beta K\left(1-\frac{u_y}{b_y}\right)$,
		\item[(iv)] $E_H^*$ if $b_x>u_x$, $b_y>u_y$, $\frac{b_xu_y}{b_y}>u_x+\beta K\left(1-\frac{u_y}{b_y}\right)$ and $R_0>1$.
	\end{itemize}
\end{theorem}
%\end{corollary}
\subsection{Model with Perfect Vertical Transmission and no Horizontal Transmission}
In this case $e=0$, $\beta=0$, and the system \eqref{continuous_1} reduces to
\begin{eqnarray}\label{continuous_3}
\frac{dX}{dt} &=& b_xX\left\{1-\frac{(X+Y)}{K}\right\}-u_xX\mbox{,}\\
\frac{dY}{dt} &=& b_yY\left\{1-\frac{(X+Y)}{K}\right\}-u_yY\mbox{.}\nonumber
\end{eqnarray}
The continuous system \eqref{continuous_3} has three equilibrium points, viz.\ $E_0^V=(0,0)$, $E_1^V=(\bar{X},0)$ and $E_2^V=(0,\bar{Y})$, where $\bar{X}=K\left(1-\frac{u_x}{b_x}\right)$ and $\bar{Y}=K\left(1-\frac{u_y}{b_y}\right)$. The existence conditions for $E_1^V$ and $E_2^V$ are $b_x>u_x$ and $b_y>u_y$, respectively. It is to be noted that no interior equilibrium does exist here. The following results are known \cite{LNEM95}.
\begin{theorem}\label{cont_theorem_3}
	System \eqref{continuous_3} is locally asymptotically stable around the equilibrium point
	\begin{itemize}
		\item[(i)] $E_0^V$ if $b_x<u_x$ and $b_y<u_y$,
		\item[(ii)] $E_1^V$ if $b_x>u_x$ and $\frac{b_y}{u_y}<\frac{b_x}{u_x}$.
		\item[(iii)] The equilibrium point $E_2^V$ is always unstable.
	\end{itemize}
\end{theorem}
\section{Discrete Models}
In this section, we construct three discrete models corresponding to the continuous models \eqref{continuous_1}, \eqref{continuous_2} and \eqref{continuous_3} following nonstandard finite difference method. The objective is to show that all the discrete models have the same dynamic properties corresponding to its continuous counterpart and the dynamics does not depend on the step size.

The NSFD procedures are based on just two fundamental rules \cite{M05,DK05,AL03}: 
\begin{itemize}
	\item [(i)] The discrete first derivative has the representation
	$$\frac{dx}{dt} \rightarrow \frac{x_{k+1}-\psi(h)x_k}{\phi(h)}\mbox{,}~h=\triangle t\mbox{,}$$
	where $\phi(h)$, $\psi(h)$ satisfy the conditions
	$\psi(h)=1+O(h^2)$, $\phi(h)=h+O(h^2)$;
	\item[(ii)] Both linear and nonlinear terms may require a nonlocal representation on the discrete computational lattice.
\end{itemize}
%	for example,
%	
%	$~~~~~~~~~~~~~~~$$x\rightarrow 2x_k-x_{k+1}$,~~~~~~~~ $x^3\rightarrow (\frac{x_{k+1}+x_{k-1}}{2})x_k^2$,
%	
%	$~~~~~~~~~~~~~~~$$x^3\rightarrow 2x_k^3-x_k^2x_{k+1}$, ~~~~$x^2\rightarrow (\frac{x_{k+1}+x_k+x_{k-1}}{3})x_k$.
%	
%\begin{definition}\label{definition3}
%	The finite difference method \eqref{Discrete model} is called positive if for any value of the step size $h$, solution of the discrete system remains positive for all positive initial values.
%\end{definition}
%\begin{definition}\label{definition4}
%	The finite difference method \eqref{Discrete model} is called elementary stable if for any value of the step size $h$, the fixed points of the difference equation are those of the differential system and the linear stability properties of each fixed point being the same for both the differential system and the discrete system.
%\end{definition}
%\begin{definition} \cite{DK06}\label{definition5}
%	A method that follows the Mickens rules (given in the Definition 3.2) and preserves the positivity of the solutions is called positive and elementary stable nonstandard (PESN) method.
%\end{definition}
For convenience, we first express the continuous system \eqref{continuous_1} as follows:
\begin{eqnarray}\label{simple_cont_1}
\frac{dX}{dt} & =&  b_xX-\frac{b_xX^2}{K}-\frac{b_xXY}{K}-u_xX-\beta XY+eY-\frac{eXY}{K}-\frac{eY^2}{K}\mbox{,}\nonumber \\
\frac{dY}{dt} & = & b_yY-\frac{b_yXY}{K}-\frac{b_yY^2}{K}-u_yY+\beta XY\mbox{.}
\end{eqnarray}
We now employ the following nonlocal approximations term wise for the system \eqref{simple_cont_1}
\begin{eqnarray}\label{NSFD_approx_1}
\left\{
\begin{array}{ll}
\frac{dX}{dt} \rightarrow \frac{X_{n+1}-X_n}{\phi_1(h)}\mbox{,}~~~~~~~~~~~~~~~~~~~~~\frac{dY}{dt} \rightarrow \frac{Y_{n+1}-Y_n}{\phi_2(h)}\mbox{,}\\
b_xX \rightarrow b_xX_n\mbox{,}~~~~~~~~~~~~~~~~~~~~~~~~~~~~~~~b_yY \rightarrow b_yY_n\mbox{,}\\
b_xX^2 \rightarrow b_xX_nX_{n+1}\mbox{,}~~~~~~~~~~~~~~~~~b_yXY \rightarrow b_yX_nY_{n+1}\mbox{,}\\
XY \rightarrow X_{n+1}Y_n\mbox{,}~~~~~~~~~~~~~~~~~~~~~~~~~b_yY^2 \rightarrow b_yY_nY_{n+1}\mbox{,}\\
u_xX \rightarrow u_xX_{n+1}\mbox{,}~~~~~~~~~~~~~~~~~~~~~~~~~~u_yY \rightarrow u_yY_{n+1}\mbox{,}\\
eY \rightarrow eY_n\mbox{,}~~~~~~~~~~~~~~~~~~~~~~~~~~~~~~~~~\beta XY \rightarrow \beta X_nY_n\mbox{,}\\
eY^2 \rightarrow e\frac{X_{n+1}Y_n^2}{X_n}\mbox{,}
\end{array}
\right.
\end{eqnarray}
where $h~(>0)$ is the step size and denominator functions are chosen as
%$$\phi_1(h)=\frac{1-\exp^{-(\frac{\beta Ku_y}{b_y})h}}{\frac{\beta Ku_y}{b_y}},
%~~~~~~\phi_2(h)=h.$$
\begin{eqnarray}\label{NSFD_deno_1}
\phi_1(h)=\frac{b_y\left\{1-exp{\left(-\frac{\beta Ku_y}{b_y}h\right)}\right\}}{\beta Ku_y}\mbox{,}
~~~~~~\phi_2(h)=h\mbox{.}
\end{eqnarray}
Note that $\phi_i(h)$, $i=1$,$2$, are positive for all $h>0$.\\
By these transformations, the continuous system \eqref{simple_cont_1} is converted to 
\begin{eqnarray}\label{discrete_system_1}\begin{split}
\frac{X_{n+1}-X_n}{\phi_1(h)} & = b_xX_n-\frac{b_x}{K}X_nX_{n+1}-\frac{b_x}{K}X_{n+1}Y_n-u_xX_{n+1}-\beta X_{n+1}Y_n+eY_n\\
&~~~-\frac{e}{K}X_{n+1}Y_n-\frac{e}{K}\frac{X_{n+1}Y_n^2}{X_n}\mbox{,}\\
\frac{Y_{n+1}-Y_n}{\phi_2(h)} & = b_yY_n-\frac{b_y}{K}X_nY_{n+1}-\frac{b_y}{K}Y_nY_{n+1}-u_yY_{n+1}+\beta X_nY_n\mbox{.}
\end{split}
\end{eqnarray}
System \eqref{discrete_system_1} can be rearranged to obtain
\begin{eqnarray}\label{NSFD_1}
X_{n+1} & = & \frac{X_n(1+\phi_1(h)b_x)+\phi_1(h)eY_n}{1+\phi_1(h)\left(\frac{b_x}{K}X_n+\frac{b_x}{K}Y_n+u_x+\beta Y_n+\frac{e}{K}Y_n+\frac{e}{K}\frac{Y_n^2}{X_n}\right)}\mbox{,}\nonumber\\
Y_{n+1} & =& \frac{Y_n\{1+\phi_2(h)(b_y+\beta X_n)\}}{1+\phi_2(h)\left(\frac{b_y}{K}X_n+\frac{b_y}{K}Y_n+u_y\right)}\mbox{,}
\end{eqnarray}
where $\phi_1(h)$ and $\phi_2(h)$ are given in \eqref{NSFD_deno_1}.\\
The model \eqref{NSFD_1} is our required discrete model corresponding to the continuous model \eqref{continuous_1}. It is to be noted that all terms in the right hand side of \eqref{NSFD_1} are positive, so solutions of the system \eqref{NSFD_1} will remain positive if they start with positive initial value. Therefore, the system \eqref{NSFD_1} is said to be positive \cite{M94}.\\
%\noindent$\bullet${\it\textbf{ Existence of fixed point}}\\
The fixed points of \eqref{NSFD_1} can be calculated by setting $X_{n+1}=X_n=X$ and $Y_{n+1}=Y_n=Y$. One thus get the fixed points as $E_0=(0,0)$, $E_1=(\bar{X},0)$, where $\bar{X}=K\left(1-\frac{u_x}{b_x}\right)$ and $E^*=(X^*,Y^*)$. Note that the equilibrium values and the existence conditions remain same as in the continuous system.\\
%\begin{eqnarray}\label{fixed_point_1}\begin{split}
%\begin{array}{l}
%b_xX-\frac{b_x}{K}X^2-\frac{b_x}{K}XY-u_xX-\beta XY+eY-\frac{e}{K}XY-\frac{e}{K}\frac{XY^2}{X} = 0,\\
%b_yY-\frac{b_y}{K}XY-\frac{b_y}{K}Y^2-u_yY+\beta XY = 0.
%\end{array}
%\end{split}
%\end{eqnarray}
%We thus get the trivial fixed point $E_0=(0,0)$ and one axial equilibrium point $E_1=(K(1-\frac{u_x}{b_x}),0)$ and the coexistence fixed point $E^*=(X^*,Y^*)$, where $X^*$ and $Y^*$ are give by (\ref{interior fixed point_1}). The trivial fixed point $E_0$ exists for all parameter values where the positive fixed point $E_1$ exits if $b_x>u_x$ and $E^*$ exists if the following three conditions hold simultaneously\\
%$~~~~~~~~~~~~~~~~~~~~~~~~~(i)b_x>u_x,~ (ii)b_y>u_y~ $ and $(iii)\frac{K}{X^*}>\frac{b_y-\beta K}{b_y-u_y}.$\\
%%$$(i)b_x>u_x,~ (ii)b_y>u_y~ and ~ (iii)\frac{K}{X^*}>\frac{b_y-\beta K}{b_y-u_y}.$$
%
%\noindent$\bullet${\it\textbf{ Stability analysis}}\\
The variational matrix of system \eqref{NSFD_1} evaluated at an arbitrary fixed point $(X,Y)$ is given by
\begin{eqnarray}\label{NSFD_jacobian_1}
J(X,Y)=
\left(
\begin{array}{cc}
a_{11}~ & ~a_{12}\\
a_{21}~ & ~a_{22}
\end{array}
\right)\mbox{,}
\end{eqnarray}
where
\begin{eqnarray}\label{NSFD_jac_element_1}
\left\{
\begin{array}{ll}
a_{11} &= \frac{1+\phi_1(h)b_x}{1+\phi_1(h)\left(\frac{b_x}{K}X+\frac{b_x}{K}Y+u_x+\beta Y+\frac{e}{K}Y+\frac{e}{K}\frac{Y^2}{X}\right)}-\frac{\{X(1+\phi_1(h)b_x)+\phi_1(h)eY\}\phi_1(h)\left(\frac{b_x}{K}-\frac{e}{K}\frac{Y^2}{X^2}\right)}{\left\{1+\phi_1(h)\left(\frac{b_x}{K}X+\frac{b_x}{K}Y+u_x+\beta Y+\frac{e}{K}Y+\frac{e}{K}\frac{Y^2}{X}\right)\right\}^2}\mbox{,}\\
a_{12} & = \frac{\phi_1(h)e}{1+\phi_1(h)\left(\frac{b_x}{K}X+\frac{b_x}{K}Y+u_x+\beta Y+\frac{e}{K}Y+\frac{e}{K}\frac{Y^2}{X}\right)}-\frac{\{X(1+\phi_1(h)b_x)+\phi_1(h)eY\}\phi_1(h)\left(\frac{b_x}{K}+\beta+\frac{e}{K}+\frac{2e}{K}\frac{Y}{X}\right)}{\left\{1+\phi_1(h)\left(\frac{b_x}{K}X+\frac{b_x}{K}Y+u_x+\beta Y+\frac{e}{K}Y+\frac{e}{K}\frac{Y^2}{X}\right)\right\}^2}\mbox{,}\\
a_{21} & = \frac{\phi_2(h)\beta Y}{1+\phi_2(h)\left(\frac{b_y}{K}X+\frac{b_y}{K}Y+u_y\right)}-\frac{Y\{1+\phi_2(h)(b_y+\beta X)\}\phi_2(h)\frac{b_y}{K}}{\left\{1+\phi_2(h)\left(\frac{b_y}{K}X+\frac{b_y}{K}Y+u_y\right)\right\}^2}\mbox{,}\\
a_{22} & = \frac{1+\phi_2(h)(b_y+\beta X)}{1+\phi_2(h)\left(\frac{b_y}{K}X+\frac{b_y}{K}Y+u_y\right)}-\frac{Y\{1+\phi_2(h)(b_y+\beta X)\}\phi_2(h)\frac{b_y}{K}}{\left\{1+\phi_2(h)\left(\frac{b_y}{K}X+\frac{b_y}{K}Y+u_y\right)\right\}^2}\mbox{.}\nonumber
\end{array}
\right.
\end{eqnarray}
Let $\lambda_{1}$ and $\lambda_{2}$ be the eigenvalues of the variational matrix \eqref{NSFD_jacobian_1} and we have the following definition \cite{E07} in relation to the stability of the system \eqref{NSFD_1}.
\begin{definition}\label{definition6}
	A fixed point $(x,y)$ of the system \eqref{NSFD_1} is called stable if $|\lambda_{1}|<1$, $|\lambda_{2}|<1$ and a source if $|\lambda_{1}|>1$, $|\lambda_{2}|>1$. It is called a saddle if $|\lambda_{1}|<1$, $|\lambda_{2}|>1$ or $|\lambda_{1}|>1$, $|\lambda_{2}|<1$ and a nonhyperbolic fixed point if either $|\lambda_{1}|=1$ or $|\lambda_{2}|=1$.
\end{definition}
\begin{lemma}[See \cite{ADS13,E07}]\label{lemma1}
	Let $\lambda_{1}$ and $\lambda_{2}$ be the eigenvalues of the variational matrix \eqref{NSFD_jacobian_1}.
	Then $|\lambda_{1}|<1$ and $|\lambda_{2}|<1$ iff
	$(i)~ 1-det(J)>0$, $(ii)~ 1-trace(J)+det(J)>0$ and $(iii)~ 0<a_{11}<1$, $0<a_{22}<1$.
\end{lemma}
One can then prove the following theorem.
\begin{theorem}\label{NSFD_theorem_1}
	System \eqref{NSFD_1} is locally asymptotically stable around the fixed point
	\begin{itemize}
		\item[(i)] $E_0$ if $b_x<u_x$ and $b_y<u_y$.
		\item[(ii)] $E_1$ if $b_x>u_x$ and $R_0<1$, where $R_0=V_0+H_0$ with $V_0=\frac{b_y}{b_x}\frac{u_x}{u_y}$, $H_0=\frac{\beta}{u_y}K\left(1-\frac{u_x}{b_x}\right)$ and it is unstable whenever $R_0>1$.
		\item[(iii)] $E^*$ if $b_x>u_x$, $b_y>u_y$, $b_y>\beta K$ and $\frac{K}{X^*}>\frac{b_y-\beta K}{b_y-u_y}$.
	\end{itemize}
\end{theorem}
\begin{proof}
	At the fixed point $E_0$, the variational matrix is given by
	\begin{eqnarray}
	J(E_0)=
	\begin{pmatrix}
	\frac{1+\phi_1(h)b_x}{1+\phi_1(h)u_x} & \frac{\phi_1(h)e}{1+\phi_1(h)u_x}\\
	& \\
	0 & \frac{1+\phi_2(h)b_y}{1+\phi_2(h)u_y}\nonumber
	\end{pmatrix}\mbox{.}
	\end{eqnarray}
	The corresponding eigenvalues are $\lambda_1=\frac{1+\phi_1(h)b_x}{1+\phi_1(h)u_x}$ and $\lambda_2=\frac{1+\phi_2(h)b_y}{1+\phi_2(h)u_y}$. Clearly $|\lambda_1|<1$ if $b_x<u_x$ and $|\lambda_2|<1$ if $b_y<u_y$, for $h>0$. Therefore, $E_0$ will be stable if $b_x<u_x$ and $b_y<u_y$ hold simultaneously.\\
	One can similarly compute the eigenvalues corresponding to the fixed point $E_1$ as $\lambda_1=\frac{1+\phi_1(h)u_x}{1+\phi_1(h)b_x}$ and $\lambda_2=\frac{1+\phi_2(h)\left\{b_y+\beta K\left(1-\frac{u_x}{b_x}\right)\right\}}{1+\phi_2(h)\left(b_y-\frac{b_yu_x}{b_x}+u_y\right)}$. Note that $|\lambda_1|<1$ whenever $E_1$ exists and $|\lambda_2|<1$ whenever $R_0<1$. Thus, $E_1$ is stable if $b_x>u_x$ and $R_0<1$.\\
	At the interior fixed point $E^*$, the variational matrix is given by
	\begin{eqnarray}
	J(E^*)=
	\begin{pmatrix}
	a_{11}^* & a_{12}^*\\
	a_{21}^* &  a_{22}^*\nonumber
	\end{pmatrix}\mbox{,}
	\end{eqnarray}
	where
	\begin{eqnarray}
	\left\{
	\begin{array}{ll}
	a_{11}^* & = 1-\frac{X^*\phi_1(h)}{G}\left\{\frac{b_xX^*}{K}+\frac{eY^*}{X^*}\left(1-\frac{X^*+Y^*}{K}\right)+\frac{eY^*}{K}\right\}\mbox{,}\\
	a_{12}^* & = \frac{\phi_1(h)X^*}{G}\left\{e\left(1-\frac{X^*+Y^*}{K}\right)-\frac{b_xX^*}{K}-\beta X^*-\frac{eY^*}{K}\right\}\mbox{,}\\
	a_{21}^* & = \frac{\phi_2(h)Y^*\beta K}{KH}-\frac{\phi_2(h)Y^*b_y}{KH}\mbox{,}\\
	a_{22}^* & = 1-\frac{\phi_2(h)b_yY^*}{KH}\mbox{,}\nonumber
	\end{array}
	\right.
	\end{eqnarray}
	with $G=X^*(1+\phi_1(h)b_x)+\phi_1(h)eY^*$ and $H=1+\phi_2(h)(b_y+\beta X^*)$.\\
	One can easily verify that $0<a_{11}^*<1$ and $0<a_{22}^*<1$. On simplifications, one can show
	\begin{eqnarray}\begin{split}
	1-det(J(E^*)) & = \frac{\phi_1(h)X^*}{KGH}\left\{b_xX^*+\frac{eY^*K}{X^*}\left(1-\frac{X^*+Y^*}{K}\right)+eY^*\right\}+\frac{\phi_1(h)\phi_2(h)X^*Y^*b_y}{KGH}\left\{\frac{b_xX^*}{Y^*}\right.\\
	&~~~\left.+\frac{eK}{X^*}\left(1-\frac{X^*+Y^*}{K}\right)^2+e+\frac{\beta b_x{X^*}^2}{b_yY^*}+\frac{2e\beta K}{b_y}\left(1-\frac{X^*+Y^*}{K}\right)+\frac{e\beta X^*}{b_y}\right.\\
	&~~~\left.+b_x\left(1-\frac{\beta X^*}{b_y}\right)+\frac{eY^*}{X^*}\left(1-\frac{\beta X^*}{b_y}\right)+\beta X^*+\beta K\left(1-\frac{X^*+Y^*}{K}\right)\right\}\\
	&~~~+\frac{\phi_2(h)b_yX^*Y^*}{KGH}\left(1-\frac{\phi_1(h)\beta Ku_y}{b_y}\right)\mbox{,}\nonumber
	\end{split}
	\end{eqnarray}
	\begin{eqnarray}\begin{split}
	1-trace(J(E^*))+det(J(E^*)) & = \frac{\phi_1(h)\phi_2(h)X^*Y^*b_y}{KGH}\left\{\frac{eK}{X^*}\left(1-\frac{X^*+Y^*}{K}\right)\left(1-\frac{u_y}{b_y}\right)+\beta X^*\left(\frac{b_x}{b_y}-1\right)\right.\\
	&~~~\left.+\frac{\beta^2KX^*}{b_y}+\frac{e\beta Y^*}{b_y}\right\}\mbox{.}\nonumber
	\end{split}	
	\end{eqnarray}
	From the existence condition, we have $\left(1-\frac{X^*+Y^*}{K}\right)=\frac{u_xX^*+\beta X^*Y^*}{b_xX^*+eY^*}>0$. Thus, $X^*+Y^*<K$, i.e., $X^*<K$. Also, from $b_y>\beta K$, we have $b_y>\beta X^*$ and $\left(1-\frac{\beta X^*}{b_y}\right)>0$. It is easy to observe that $\phi_1(h)<\frac{b_y}{\beta Ku_y}$. Thus, $1-det(J(E^*))>0$ and $1-trace(J(E^*))+det(J(E^*))>0$. Hence $E^*$ is locally asymptotically stable whenever it exists. This completes the theorem.
\end{proof}
\begin{remark}
	It is interesting to note that the dynamic properties of the discrete system \eqref{NSFD_1} are identical with its continuous counterpart \eqref{continuous_1}. So the discrete model is dynamically consistent. The stability of the fixed points also does not depend on the step size. Since all solutions of the discrete model \eqref{NSFD_1} remain positive when starts with positive initial value, there is no possibility of numerical instabilities and the model will not show any spurious dynamics.
\end{remark}
\subsection{Discrete Model for Horizontal and Perfect Vertical Transmissions}
Here we rewrite the continuous model \eqref{continuous_2} as
\begin{eqnarray}\label{simple_cont_2}
\frac{dX}{dt} & = & b_xX-\frac{b_xX^2}{K}-\frac{b_xXY}{K}-u_xX-\beta XY\mbox{,}\nonumber\\
\frac{dY}{dt} & = & b_yY-\frac{b_yXY}{K}-\frac{b_yY^2}{K}-u_yY+\beta XY\mbox{.}
\end{eqnarray}
Now we employ the same nonlocal approximations \eqref{NSFD_approx_1} with $e=0$ term wise to have the following system: 
\begin{eqnarray}\label{discrete_system_2}
\frac{X_{n+1}-X_n}{\phi_1(h)} & = & b_xX_n-\frac{b_x}{K}X_nX_{n+1}-\frac{b_x}{K}X_{n+1}Y_n-u_xX_{n+1}-\beta X_{n+1}Y_n\mbox{,}\nonumber\\
\frac{Y_{n+1}-Y_n}{\phi_2(h)} & = & b_yY_n-\frac{b_y}{K}X_nY_{n+1}-\frac{b_y}{K}Y_nY_{n+1}-u_yY_{n+1}+\beta X_nY_n\mbox{.}
\end{eqnarray}
The required discrete model is obtained after simplification as follows:  \begin{eqnarray}\label{NSFD_2}
X_{n+1} & = & \frac{X_n(1+\phi_1(h)b_x)}{1+\phi_1(h)\left(\frac{b_x}{K}X_n+\frac{b_x}{K}Y_n+u_x+\beta Y_n\right)}\mbox{,}\nonumber\\
Y_{n+1} & = & \frac{Y_n\{1+\phi_2(h)(b_y+\beta X_n)\}}{1+\phi_2(h)\left(\frac{b_y}{K}X_n+\frac{b_y}{K}Y_n+u_y\right)}\mbox{,}
\end{eqnarray}
where $\phi_1(h)$ and $\phi_2(h)$ have the same expression as in \eqref{NSFD_deno_1}. It is worth mentioning that the discrete model \eqref{NSFD_2} is positive.\\
One can find the same four fixed points of \eqref{NSFD_2} as it were in the continuous case. The stability properties of each fixed point are presented in the following theorem.
\begin{theorem}\label{NSFD_theorem_2}
	The system \eqref{NSFD_2} is stable around the fixed point
	\begin{itemize}
		\item[(i)] $E_0^H=(0,0)$ if $b_x<u_x$ and $b_y<u_y$.
		\item[(ii)] $E_1^H=(\bar{X},0)$ if $b_x>u_x$ and $R_0<1$, where $\bar{X}=K\left(1-\frac{u_x}{b_x}\right)$ and $R_0=\frac{b_y}{b_x}\frac{u_x}{u_y}+\frac{\beta}{u_y}\bar{X}$.
		\item[(iii)] $E_2^H=(0,\bar{Y})$ if $b_y>u_y$ and $\frac{b_xu_y}{b_y}<u_x+\beta K\left(1-\frac{u_y}{b_y}\right)$, where $\bar{Y}=K\left(1-\frac{u_y}{b_y}\right)$.
		\item[(iv)] $E_H^*$ if $b_x>u_x$, $b_y>u_y$, $\frac{b_xu_y}{b_y}>u_x+\beta K\left(1-\frac{u_x}{b_x}\right)$ and $R_0>1$.
	\end{itemize}
\end{theorem}
\subsection{Discrete Model for Perfect Vertical and no Horizontal Transmission}
For convenience, we first express the continuous system \eqref{continuous_3} as
\begin{eqnarray}\label{simple_cont_3}
\frac{dX}{dt} & = & b_xX-\frac{b_xX^2}{K}-\frac{b_xXY}{K}-u_xX\mbox{,}\\
\frac{dY}{dt} & = & b_yY-\frac{b_yXY}{K}-\frac{b_yY^2}{K}-u_yY\mbox{.}\nonumber
\end{eqnarray}
In this case, we consider the nonlocal approximations \eqref{NSFD_approx_1} with $e=0$, $\beta=0$. Note that here $\phi_1(h)\rightarrow h$ when $\beta \rightarrow 0$. Then the system \eqref{simple_cont_3} reads
\begin{eqnarray}\label{discrete_system_3}
\frac{X_{n+1}-X_n}{h} & = & b_xX_n-\frac{b_x}{K}X_nX_{n+1}-\frac{b_x}{K}X_{n+1}{Y_n}-u_xX_{n+1}\mbox{,}\\
\frac{Y_{n+1}-Y_n}{h} & = & b_yY_n-\frac{b_y}{K}X_nY_{n+1}-\frac{b_y}{K}Y_{n}Y_{n+1}-u_yY_{n+1}\mbox{.}\nonumber
\end{eqnarray}
On simplifications, we obtain our desired discrete model as
\begin{eqnarray}\label{NSFD_3}
X_{n+1} & = & \frac{X_n(1+hb_x)}{1+h\left(\frac{b_x}{K}X_n+\frac{b_x}{K}Y_n+u_x\right)}\mbox{,}\\
Y_{n+1} & = & \frac{Y_n(1+hb_y)}{1+h\left(\frac{b_y}{K}X_n+\frac{b_y}{K}Y_n+u_y\right)}\mbox{.}\nonumber
\end{eqnarray}
This system also does not contain any negative terms, so solutions remain positive for all step size as long as initial values are positive.

As in the continuous system \eqref{continuous_3}, the discrete system \eqref{NSFD_3} has same three fixed points. The stability of each fixed point can be proved similarly and has been summarized in the following theorem.
\begin{theorem}\label{NSFD_theorem_3}
	The system \eqref{NSFD_3} is stable around the fixed point
	\begin{itemize}
		\item[(i)] $E_0^V=(0,0)$ if $b_x<u_x$ and $b_y<u_y$.
		\item[(ii)] $E_1^V=(\bar{X},0)$ if $b_x>u_x$ and $\frac{b_y}{u_y}<\frac{b_x}{u_x}$.
		\item[(iii)] The fixed point $E_2^V=(0,\bar{Y})$ is always unstable.
	\end{itemize}
\end{theorem}
\section{Numerical Simulations}
\begin{figure}[H]
	\begin{center}
		\includegraphics[height=4.2in, width=5.5in]{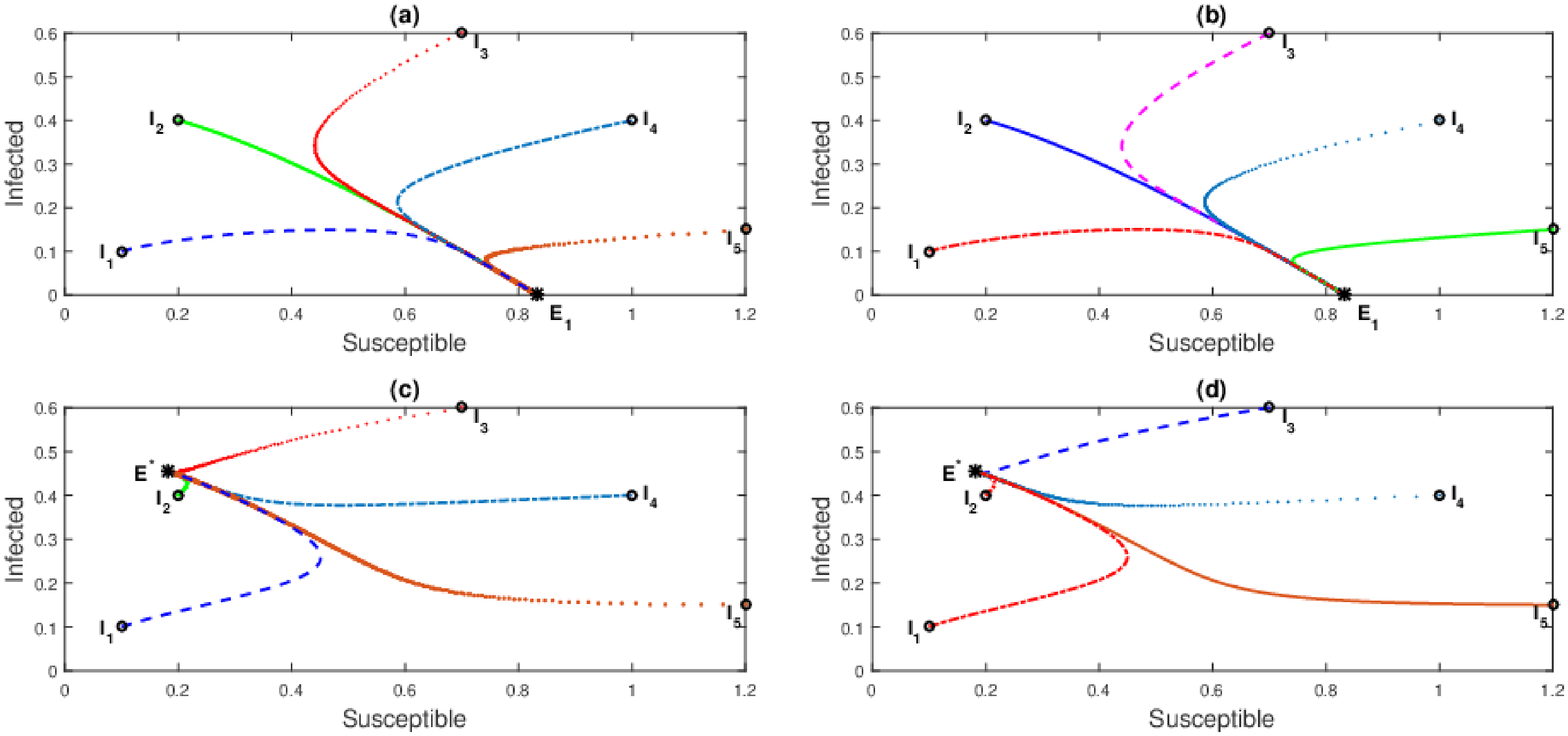}
		\caption{Phase portraits of the continuous system \eqref{continuous_1} (left panel) and discrete system \eqref{NSFD_1} (right panel). Figs (a) and (b) show that all solutions converge to the disease free equilibrium point $E_1=(0.8333,0)$ for $\beta=0.1$. Figs (c) and (d) depict that all solutions converge to the endemic equilibrium point $E^*=(0.1818,0.4545)$ for $\beta=0.3$. Other parameters are $b_x=0.6$, $b_y=0.4$, $u_x=0.1$, $u_y=0.2$, $K=1$, $e=0.02$ as in \cite{LNEM95}. Step size for the discrete model is considered as $h=0.1$.}
		\label{phase_plane_1}
	\end{center}
\end{figure}

In this section, we present some numerical simulations to validate the similar qualitative behavior of our discrete models with its corresponding continuous models. For this, we consider the same parameter set as in Lipsitch et al.\ \cite{LNEM95}: $b_x=0.6$, $b_y=0.4$, $u_x=0.1$, $u_y=0.2$, $K=1$, $e=0.02$. We consider different initial values $I_1=(0.1,0.1)$, $I_2=(0.2,0.4)$, $I_3=(0.7,0.6)$, $I_4=(1,0.4)$ and $I_5=(1.2,0.15)$ for both continuous and discrete systems. Step size $h=0.1$ is kept fixed in all simulations for the discrete systems. If $\beta$ takes the value $0.1$, the parameter set satisfies conditions of Theorems \ref{cont_theorem_1}(ii) and \ref{NSFD_theorem_1}(ii). In this case, all solutions starting from different initial points converge to the infection free point $E_1=(0.8333,0)$ in case of both the continuous system \eqref{continuous_1} (Fig. \ref{phase_plane_1}(a)) and the discrete system \eqref{NSFD_1} (Fig. \ref{phase_plane_1}(b)). For $\beta=0.3$, conditions of Theorems \ref{cont_theorem_1}(iii) and \ref{NSFD_theorem_1}(iii) are satisfied and all solution trajectories reach to the coexistence equilibrium point $E^*=(0.1818,0.4545)$ for both the systems as shown in Fig. \ref{phase_plane_1}(c)--\ref{phase_plane_1}(d). These figures indicate that the behavior of the continuous system \eqref{continuous_1} and the discrete system \eqref{NSFD_1} are qualitatively similar.
\begin{figure}[H]
	\begin{center}
		\includegraphics[height=3.5in, width=5.5in]{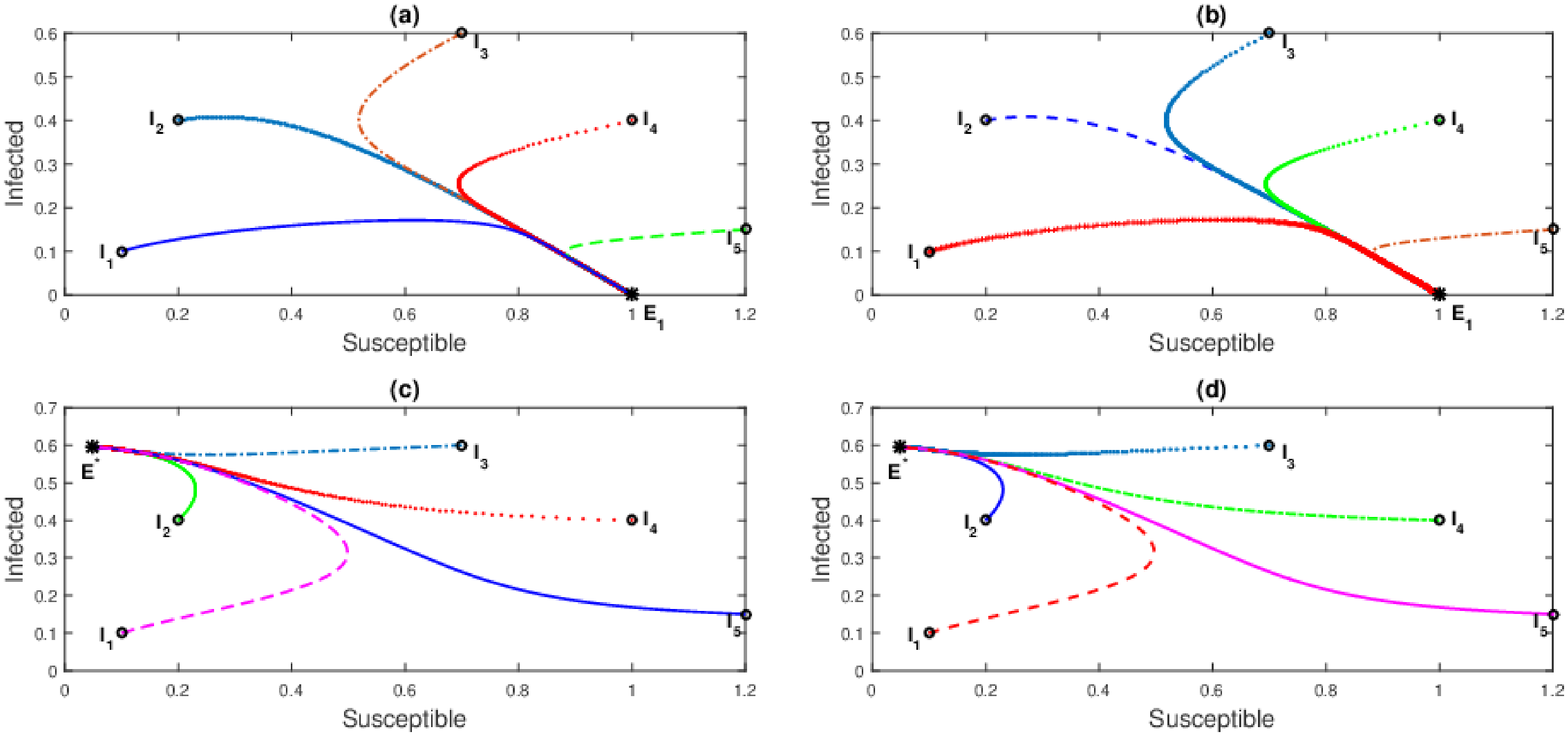}
		\includegraphics[height=1.75in, width=5.5in]{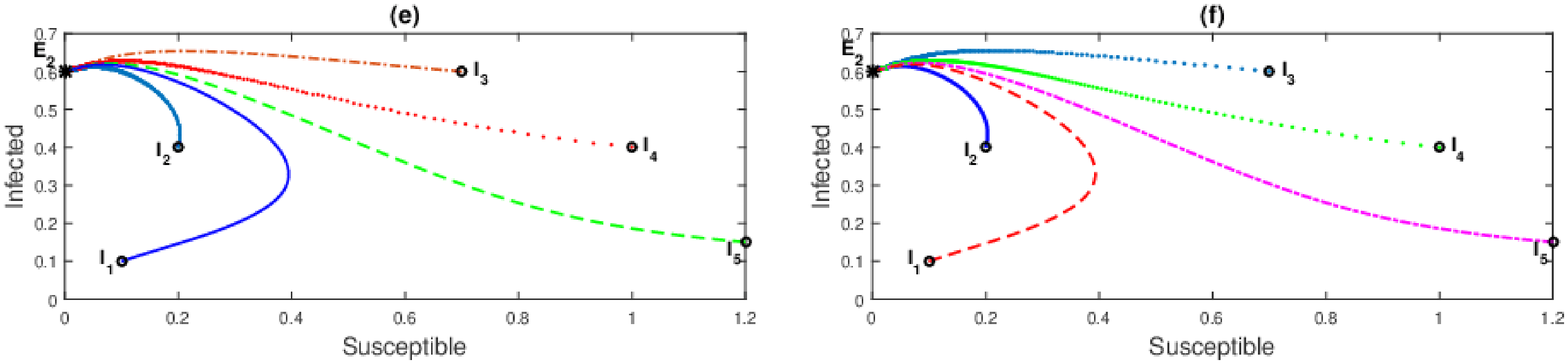}
		\caption{Phase portraits of the continuous system \eqref{continuous_2} (left panel) and discrete system \eqref{NSFD_2} (right panel). Figs. (a) and (b) show that all solutions converge to the disease free point $E_1^H=(1,0)$ for $\beta=0.1$. Figs. (c) and (d) depict that all solutions converge to the endemic point $E^*=(0.0476,0.5952)$ for $\beta=0.3$. Figs. (e) and (f) show that all solutions converge to the susceptible free point $E_2^H=(0,0.6)$ for $\beta=0.42$. Other parameters are $b_x=0.6$, $b_y=0.4$, $u_x=0.1$, $u_y=0.2$, $K=1.2$ as in \cite{LNEM95}. Step size for the discrete model is considered as $h=0.1$.}
		\label{phase_plane_2}	
	\end{center}
\end{figure}

To show dynamic consistency of the continuous system \eqref{continuous_2} and discrete system \eqref{NSFD_2}, we plotted the phase portraits of both systems in Fig. \ref{phase_plane_2}. We considered the same initial points, the same set of parameter values as in \cite{LNEM95} with $e=0$ and the same step size as in Fig. \ref{phase_plane_1}.The conditions of Theorem \ref{cont_theorem_2}(ii) and Theorem \ref{NSFD_theorem_2}(ii) are satisfied when $\beta=0.1$. In this case all solutions of both the systems converge to the point $E_1^H=(1,0)$ (Figs. \ref{phase_plane_2}(a)--\ref{phase_plane_2}(b)). For $\beta=0.3$, conditions of Theorem \ref{cont_theorem_2}(iv) and Theorem \ref{NSFD_theorem_2}(iv) are satisfied. Consequently, all solutions reach to the interior point $E_H^*=(0.0476,0.5952)$ (Figs. \ref{phase_plane_2}(c)--\ref{phase_plane_2}(d)). If we take $\beta=0.42$ then all conditions of Theorem \ref{cont_theorem_2}(iii) and Theorem \ref{NSFD_theorem_2}(iii) are satisfied. All solutions in this case converge to the susceptible free equilibrium point $E_2^H=(0,0.6)$ in both cases (Figs. \ref{phase_plane_2}(e)--\ref{phase_plane_2}(f)).

To observe dynamical consistency of the discrete system \eqref{NSFD_3} with its corresponding continuous system \eqref{continuous_3}, we plotted phase diagrams of both systems in Fig. \ref{phase_plane_3}. The same parameter set as in \cite{LNEM95} with $e=0$, $\beta=0$ was considered and the initial points, step size remained unchanged. Phase portraits of the continuous system (Fig. \ref{phase_plane_3}(a)) and that of the discrete system (Fig. \ref{phase_plane_3}(b)) show that all solutions reach to the infection free point $E_1^V=(1,0)$, indicating the dynamic consistency of both systems.
\begin{figure}[H]
	\begin{center}
		\includegraphics[height=2.0in, width=5.5in]{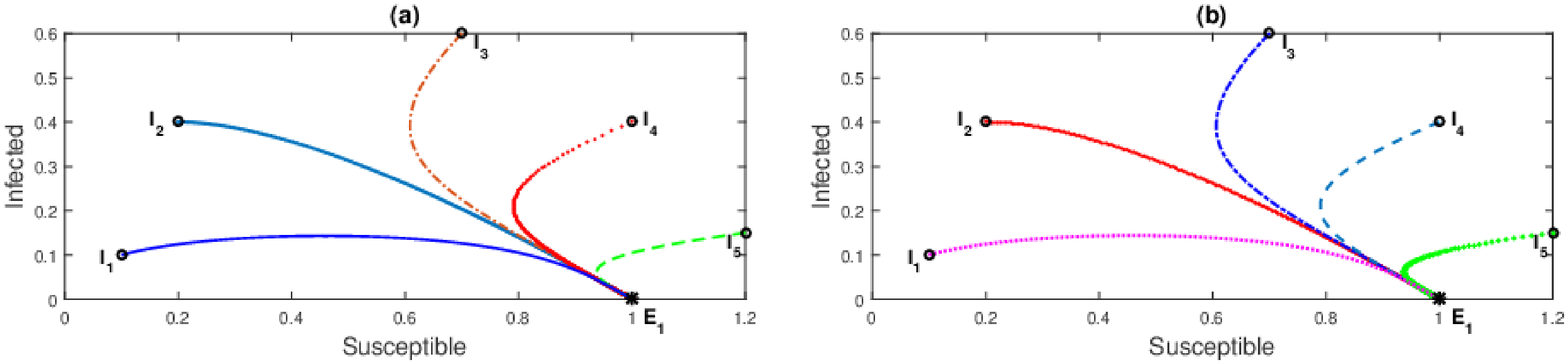}
		\caption{Phase portrait of the continuous system \eqref{continuous_3} (Fig. a) and that of the discrete system \eqref{NSFD_3} (Fig. b) indicate that all solutions converge to the infection free point $E_1^V=(1,0)$ in each case. The parameters are $b_x=0.6$, $b_y=0.4$, $u_x=0.1$, $u_y=0.2$ and $K=1.2$ as in \cite{LNEM95}. Step size for the discrete model is considered as $h=0.1$.}
		\label{phase_plane_3}
	\end{center}
\end{figure}
\section{Summary}
We here considered a continuous time epidemic model, where infection spreads through imperfect vertical transmission and horizontal transmission in a density dependent asexual host population. Stability of different equilibrium points are presented with respect to the basic reproduction number and relative birth \& death rates of susceptible \& infected hosts. A discrete version of the continuous system is constructed following nonlocal approximation technique and its dynamics has been shown to be identical with that of the continuous system. The proposed discrete model is shown to be positive, implying that its solutions remains positive for all future time whenever it starts with positive initial value. The dynamics of the discrete model have been shown to be independent of the step size. Our simulation results also show dynamic consistency of the discrete models with its corresponding continuous model. Two submodels of the general discrete model have also been shown to have the identical dynamics with their continuous continuous counterparts.
\section*{Acknowledgements} {Research of P. Saha is supported by CSIR; F. No: 09/096(0909)/2017-EMR-I and research of N. Bairagi is supported by SERB, DST; F. No: MTR/2017/000032.
	\bibliographystyle{plain}
	%\bibliography{bib}

\begin{thebibliography}{00}
		
		\bibitem{LNEM95}
		Lipsitch, M., Nowak, M.A., Ebert, D. and May, R.M., 1995. \emph{The  population dynamics of vertically and horizontally transmitted parasites}, Biological Sciences Vol.~260, Issue 1359, 321--327.
		
		
		\bibitem{HB72} Holmes, J.C., Bethel, W.M., 1972. \emph{Modification of intermediate host behavior by parasites. In: Canning, E.V., Wright, C.A. (Eds.), Behavioral Aspects of Parasite Transmission}, Suppl. I to Zool. f. Linnean Soc. 51, 123--149.
		
		\bibitem{LM96} Lafferty, K.D., Morris, A.K., 1996. \emph{Altered behaviour of parasitized killfish increases susceptibility to predation by bird final hosts}, Ecology 77, 1390--1397.
		
		\bibitem{M89}
		Mickens, R.E., 1989. \emph{Exact solutions to a finite-difference model of a nonlinear reaction-advection equation: Implications for numerical analysis}, Numer. Methods Partial Diff. Eq.,5, 313--325.
		
		\bibitem{MET03}
		Moghadas, S.M., Alexander, M.E., Corbett, B.D. and Gumel, A.B., 2003. \emph{A Positivity-preserving Mickens-type discretization of an epidemic model}, J. Diff. Equ. Appl., 9, 1037--1051.
		
		\bibitem{BB17a}
		Biswas, M. and Bairagi, N., 2017. \emph{Discretization of an eco-epidemiological model and its dynamic consistency}, J. Diff. Equ. Appl., 23(5), 860--877.
		
		\bibitem{SI10}
		Sekiguchi, M. and Ishiwata, E., 2010. \emph{Global dynamics of a discretized SIRS epidemic model with time delay}, J. Math. Anal. Appl., 371, 195--202.
		
		\bibitem{BB17b}
		Biswas, M. and Bairagi, N., 2016. \emph{Dynamic consistency in a predatorprey model with habitat complexity: Nonstandard versus standard finite difference methods}, Int. J. Diff. Equ. Appl., 11(2), 139--162.
		
		\bibitem{RL13}
		Roege, L.I. and Lahodny, G., 2013. \emph{Dynamically consistent discrete Lotka-Volterra competition systems}, J. Diff. Equ. Appl., 19, 191--200.
		
		\bibitem{BET17}
		Biswas, M., Lu, X., Cao, X. and Bairagi, N., 2018. \emph{On the dynamic consistency of a delay-induced discrete predator-prey model}, Int. J. Diff. Equ.,(To Appear).
		
		\bibitem{G12}
		Gabbriellini, G., 2012. \emph{Nonstandard finite difference scheme for mutualistic interaction description}, Int. J. Differ. Equ., 9, 147--161.
		
		\bibitem{BB17}
		Biswas, M. and Bairagi, N., 2017. \emph{A predator-prey model with Beddington-DeAngelis Functional Response: A non-standard finite-difference method}, J. Diff. Equ. Appl. doi.org/10.1080/10236198.2017.1304544, 2017.
		
		\bibitem{M05}
		Mickens, R.E., 2005. \emph{Dynamic consistency: a fundamental principle for constructing NSFD schemes for differential equations}, J. Diff. Equ. Appl., 11, 645--653.
		
		\bibitem{DK05}
		Dimitrov, D.T., Kojouharov, H.V., 2005. \emph{Nonstandard finite-difference schemes for geneal two-dimensional autonomus dynamical systems}, Appl. Maths. Letters, 18, 769--774.
		
		\bibitem{AL03}
		Anguelov, R., Lubuma, J.M.S., 2003. \emph{Nonstandard Finite Difference Method by Nonlocal Approximation}, Math. and Compu. Simulation, 61, 465--475.
		
		\bibitem{M94}
		Mickens, R.E., 1994. \emph{Nonstandard finite difference models of differential equations}, World Scientific.
		
		\bibitem{ADS13}
		Alawia, R., Darti, I., Suryanto, A., 2013. \emph{Stability and bifurcation analysis of discrete partial dependent predator-prey model with delay}, App. Math. Sc., 7, 4403--4413.
		
		\bibitem{E07}
		Elaydi, S.N., 2007. \emph{Discrete chaos with applications in science and engineering}, Chapman and Hall/CRC, New York.
		
		
	\end{thebibliography}

\end{document}